\documentclass[natbib,graybox]{svmult}
\usepackage[latin1]{inputenc}

\usepackage{mathptmx,helvet,courier} 
\usepackage[bottom]{footmisc} 

\usepackage{amssymb,amsmath,amsfonts}
\usepackage{bm, bbm}
\usepackage{caption,xspace,enumerate}

\usepackage{bigfoot}

\bibpunct{(}{)}{~;}{a}{,}{,}

\makeatletter
\spn@wtheorem{assum}{Assumption}{\bfseries}{\rmfamily}
\spn@wtheorem{claaa}{Claim}{\itshape}{\rmfamily}
\makeatother

\def\abstractname{Abstract.\ }

\makeatletter
\renewcommand\section{\@startsection{section}{1}{\z@}{-20pt}{8pt}%
  {\normalfont\secsize\secstyle \rightskip=\z@ \@plus 8em\pretolerance=10000 }}
\makeatother

\newcommand{\RR}{{\mathbb{R}}}
\newcommand{\NNN}{{\mathbb{N}}}
\newcommand{\HH}{{\mathcal{H}}}
\newcommand{\GG}{{\mathcal{G}}}
\newcommand{\FF}{{\mathcal{F}}}

\newcommand{\XX}{{\mathbb{X}}}
\renewcommand{\Pr}{\mathsf{P}}
\DeclareMathOperator{\var}{var}

\DeclareMathOperator{\EE}{\mathsf{E}}
\DeclareMathOperator{\Span}{{\rm span}}

\newcommand \dotvar {\bm{\cdot}}

\newcommand \scd[1] {\left\langle\, #1 \,\right\rangle}

\newcommand \defeq  {\mathrel{\mathop:}=}

\providecommand{\abs}[1]{\lvert#1\rvert}
\providecommand{\norm}[1]{\big\lVert#1\big\rVert}
\providecommand{\ns}[1]{\lVert#1\rVert}

\renewcommand{\hat}{\widehat}
\newcommand{\x}{\underline{x}}

\newcommand \NEBi {\ref{NEB-prop}.\ref{NEB-i}\xspace}
\newcommand \NEBii {\ref{NEB-prop}.\ref{NEB-ii}\xspace}

\newcommand \PredMes[2]   {\lambda_{\,#1,#2}}
\newcommand \PredCoeff[3] {\lambda^{#1}(#3;\x_{#2})}

\begin{document}

\title*{%
  Pointwise consistency of the kriging predictor
  with known mean and covariance functions}

\titlerunning{Pointwise consistency of kriging}

\author{Emmanuel Vazquez \and Julien Bect}

\institute{%
  Emmanuel Vazquez \and Julien Bect \at SUPELEC, 
  3 rue Joliot-Curie, 91192 Gif-sur-Yvette, France \\
  \email{{emmanuel.vazquez@supelec.fr} and {julien.bect@supelec.fr}}}
\maketitle

\newcommand \myabstract {This paper deals with several issues related
  to the pointwise consistency of the kriging predictor when the mean
  and the covariance functions are known. These questions are of
  general importance in the context of computer experiments. The
  analysis is based on the properties of approximations in reproducing
  kernel Hilbert spaces.  We fix an erroneous claim of Yakowitz and
  Szidarovszky (J. Multivariate Analysis, 1985) that the kriging
  predictor is pointwise consistent for all continuous sample paths under
  some assumptions.}

\abstract{\myabstract} \abstract*{\myabstract}


\section{Introduction}
\label{sec:intro}

The domain of \emph{computer experiments} is concerned with making
inferences about the output of an expensive-to-run numerical
simulation of some physical system, which depends on a vector of
factors with values in $\XX\subseteq \RR^d$. The output of the
simulator is formally an unknown function $f:\XX \to\RR$. For example,
to comply with ever-increasing standards regarding pollutant
emissions, numerical simulations are used to determine the level of
emissions of a combustion engine as a function of its design
parameters \citep{villetex:2008:ofc}. The emission of pollutants by
an engine involves coupled physical phenomena whose numerical
simulation by a finite-element method, for a fixed set of design
parameters of the engine, can take several hours on high-end
servers. It then becomes very helpful to collect the answers already
provided by the expensive simulator, and to construct from them a
simpler computer model, that will provide approximate but cheaper
answers about a quantity of interest. This approximate model is often
called a \emph{surrogate}, or a~\emph{metamodel}, or
an~\emph{emulator} of the actual simulator $f$. The quality of the
answers given by the approximate model depends on the quality of the
approximation, which depends, in turn and in part, on the choice of the
evaluation points of $f$, also called \emph{experiments}. The choice
of the evaluation points is usually called the \emph{design of
  experiments}.  Assuming that~$f$ is continuous, it is an important
question to know whether the approximate model behaves consistently,
in the sense that if the evaluation points~$x_n$ are chosen
sequentially in such a way that a given point $x\in \XX$ is an
accumulation point of~$\{x_n,\ n\geq 1\}$, then the approximation at
$x$ converges to $f(x)$.

Since the seminal paper of~\cite{sacks:89:dace}, \emph{kriging} has been
one of the most popular methods for building approximations in the
context of computer experiments \citep[see, e.g.,][]{santner:2003:dace}.
In the framework of kriging, the unknown function~$f$ is seen as a
sample path of a stochastic process~$\xi$, which turns the problem of
approximation of~$f$ into a prediction problem for the process~$\xi$. In
this paper, we shall assume that the mean and the covariance functions
are known. Motivated by the analysis of the \emph{expected improvement}
algorithm \citep{vazquez:2009:ei}, a popular kriging-based optimization
algorithm, we discuss several issues related to the pointwise
consistency of the kriging predictor, that is, the convergence of the
kriging predictor to the true value of~$\xi$ at a fixed point~$x \in \XX$.
These issues are barely documented in the literature, and we believe
them to be of general importance for the asymptotic analysis of
sequential design procedures based on kriging.


The paper is organized as follows.  Section \ref{sec:problems}
introduces notation and various formulations of pointwise consistency,
using the reproducing kernel Hilbert space (RKHS) attached to $\xi$.
Section~\ref{sec:NEB-prop} investigates whether $L^2$-pointwise
consistency at~$x$ can hold when~$x$ is not in the adherence of the
set~$\{x_n, n\geq 1\}$. Conversely, assuming that~$x$ is in the adherence,
Section~\ref{sec:falsethm} studies the set of sample paths~$f =
\xi(\omega,\dotvar)$ for which pointwise consistency holds. In particular, we
fix an erroneous claim of \cite{yakowitz85:_compar_krigin}---namely,
that the kriging predictor is pointwise consistent for all continuous
sample paths under some assumptions.

\section{Several formulations of pointwise consistency}
\label{sec:problems}

Let $\xi$ be a second-order process defined on a probability
space~$( \Omega, \mathcal{A}, \Pr)$, with parameter $x \in \XX \subseteq \RR^d$.
Without loss of generality, it will be assumed that the mean of~$\xi$ is
zero and that $\XX = \RR^d$. The covariance function of~$\xi$ will be
denoted by $k(x,y) \defeq \EE\left[ \xi(x) \xi(y) \right]$, and the following
assumption will be used throughout the paper:
\begin{assum}
  \label{assum:cont}
  The covariance function~$k$ is continuous.
\end{assum}
The kriging predictor of~$\xi(x)$, based on the observations $\xi(x_i)$,
$i=1,\ldots, n$, is the orthogonal projection
\begin{equation}
  \label{eq:def-krig-pred}
  \hat \xi( x; \x_n ) \;\defeq\; \sum_{i=1}^n \PredCoeff{i}{n}{x}\, \xi(x_i)  
\end{equation}
of $\xi(x)$ onto $\Span\{\xi(x_i),i=1,\ldots ,n\}$.  The variance of the prediction
error, also called the \emph{kriging variance} in the literature of
geostatistics \citep[see, e.g.,][]{chiles99}, or the \emph{power
  function} in the literature of radial basis functions~\citep[see,
e.g., ][]{wu93:_local}, is
\begin{align*}
  \label{eq:def-krig-var}
  \sigma^2(x; \x_n)
  \;&\defeq\; \var \left[ \xi(x)- \hat\xi(x;\x_n) \right] \\
  &=\; k(x,x) - \sum_i \PredCoeff{i}{n}{x}\, k(x,x_i) \,.
\end{align*}

For any $x\in \RR^d$, and any sample path $f=\xi(\omega,\dotvar)$,
$\omega\in\Omega$, the values
$\xi(\omega, x)=f(x)$ and~$\hat\xi(\omega, x;\x_n)$ can be seen as the
result of the application of an evaluation functional to $f$. More precisely, let
$\delta_x$ be the Dirac measure at $x\in\RR^d$, and let $\PredMes{n}{x}$
denote the measure with finite support defined by $\PredMes{n}{x} \defeq
\sum_{i=1}^n \PredCoeff{i}{n}{x}\, \delta_{x_i}$.  Then, for all $\omega\in\Omega$,
$\xi(\omega,x) = \scd{ \delta_x,\, f }$ and
$\hat \xi(\omega,x;\x_n) = \scd{ \PredMes{n}{x},\, f }$.
Pointwise consistency at~$x\in\RR^d$, defined in Section~\ref{sec:intro} as the
convergence of~$\hat \xi(\omega,x;\x_n)$ to~$\xi(x)$, can thus be seen as the
convergence of~$\PredMes{n}{x}$ to~$\delta_x$ in some sense.

Let~$\HH$ be the RKHS of functions generated by~$k$, and $\HH^*$ its
dual space. Denote by~$(\cdot, \cdot)_{\HH}$ (resp. $(\cdot,
\cdot)_{\HH^*}$) the inner product of~$\HH$ (resp. $\HH^*$), and
by~$\ns{\cdot}_{\HH}$ (resp.  $\ns{\cdot}_{\HH^*}$) the corresponding
norm. It is well-known \citep[see, e.g.,][]{wu93:_local} that
\begin{equation*}
  \norm{ \delta_x - \PredMes{n}{x} }^2_{\HH^*}
  \;=\; \norm{ k(x,\cdot) - {\sum}_i\, 
    \PredCoeff{i}{n}{x}\, k(x_i,\cdot) }^2_{\HH} 
  \;=\; \sigma^2(x;\x_n) \,.
\end{equation*}
Therefore, the convergence
$\PredMes{n}{x} \to \delta_x$ holds strongly in~$\HH^*$ if and only if the
kriging predictor is $L^2(\Omega, \mathcal{A}, \Pr)$-consistent
at~$x$; that is, if $\sigma^2(x;\x_n)$ converges to zero. 
Since~$k$ is continuous, it is easily seen that $\sigma^2(x;\x_n) \to 0$ as soon as~$x$ is
adherent to~$\{ x_n, n\geq 1\}$. Indeed,
$$
\sigma^2(x,\x_n) \leq \EE[(\xi(x) - \xi(x_{\varphi_n}))^2] = k(x,x)
  + k(x_{\varphi_n},x_{\varphi_n}) - 2k(x,x_{\varphi_n}),
$$
with
$(\varphi_n)_{n\in\NNN}$ a non-decreasing sequence such that
$\forall n\geq 1$, $\varphi_n \leq n$ and $x_{\varphi_n} \to x$.
As explained by
\citet{vazquez:2009:ei}, it is sometimes important to work with
covariance functions such that the converse holds. That leads to
our first open issue, which will be discussed in
Section~\ref{sec:NEB-prop}:

\begin{problem}
  Find necessary and sufficient conditions on a
  continuous covariance~$k$ such that $\sigma^2(x;\x_n) \to 0$ implies
  that $x$ is adherent to~$\{ x_n, n\geq 1\}$.
\end{problem}

Moreover, since
strong convergence in~$\HH^*$ implies weak convergence in~$\HH^*$, we
have
\begin{equation}
  \label{weak-Hstar-const}
  \lim_{n\to \infty} \sigma^2(x;\x_n) = 0 \;\implies\; \forall f \in \HH\,,
  \quad\lim_{n\to \infty} \scd{\PredMes{n}{x}, f} 
  = \scd{ \delta_x,\, f } = f(x)\,. 
\end{equation}
Therefore, if~$x$ is adherent to~$\{ x_n,\, n\geq 1\}$, pointwise consistency
holds for all sample paths $f \in \HH$. However, this result is not
satisfying from a Bayesian point of view since $\Pr\{\xi\in \HH\} = 0$ if $\xi$
is Gaussian \citep[see, e.g.,][Driscoll's
theorem]{lukic01:_stoch_hilber}. In other words, modeling~$f$ as a
Gaussian process means that $f$ cannot be expected to belong
to~$\HH$. This leads to our second problem:
\begin{problem}
  For a given covariance function~$k$, describe the set of functions
  $\GG$ such that, for all sequences $(x_n)_{n \geq 1}$ in~$\RR^d$
  and all $x \in \RR^d$,
  \begin{equation}
    \label{GG-const}
    \lim_{n\to \infty} \sigma^2(x;\x_n) = 0 \;\implies\; \forall f \in \GG\,,
    \quad\lim_{n\to \infty} \scd{\PredMes{n}{x}, f} = f(x)\,.
  \end{equation}
\end{problem}

An important question related to this problem, to be discussed in
Section~\ref{sec:falsethm}, is to know whether the set~$\GG$ contains
the set~$C(\RR^d)$ of all continuous functions. Before proceeding, we
can already establish a result which ensures that considering the
kriging predictor is relevant from a Bayesian point of view.
\begin{theorem}
  \label{thm:as-consistency}
  If $\xi$ is Gaussian, then $\{\xi\not\in\GG\}$ is $\Pr$-negligible.
\end{theorem}
\begin{proof}
  If $\xi$ is Gaussian, it is well-known that $\hat \xi(x;\x_n) \;=\;
  \EE[\xi(x) \mid \FF_n]$ a.s., where $\FF_{n}$ denotes the $\sigma$-algebra
  generated by $\xi(x_1)$, \ldots, $\xi(x_n)$.  Note that $\left( \EE[\xi(x) \mid
    \FF_n] \right)$ is an $L^2$-bounded martingale sequence and
  therefore converges, a.s.  and in $L^2$-norm, to a random variable
  $\xi_{\infty}$ \citep[see, e.g.,][]{Williams:1991:PwM}.\qed
\end{proof}

\section{Pointwise consistency in $L^2$-norm and the No-Empty-Ball
  property}
\label{sec:NEB-prop}

The following definition has been introduced by
\citet{vazquez:2009:ei}:

\begin{definition}
  \label{NEB-prop}
  A random process~$\xi$ has the No-Empty-Ball
  (NEB) property if, for all sequences $(x_n)_{n \geq 1}$ in~$\RR^d$ and
  all $x \in \RR^d$, the following assertions are equivalent:
  \begin{enumerate}[i)]
  \item\label{NEB-i} $x$ is an adherent point of the set $\{ x_n,\, n
    \geq 1 \}$,
  \item\label{NEB-ii} $\sigma^2(x, \x_n) \to 0$ when $n \to +\infty$.
  \end{enumerate}
\end{definition}
The NEB property implies that there can be no empty ball centered at~$x$
if the prediction error at~$x$ converges to zero---hence the
name. Since~$k$ is continuous, 
the implication \NEBi $\Rightarrow$ \NEBii is
true. Therefore, Problem~$1$ amounts to finding necessary and
sufficient conditions on~$k$ for~$\xi$ to have the NEB property.

Our contribution to the solution of Problem~$1$ will be
twofold. First, we shall prove that the following assumption,
introduced by \cite{yakowitz85:_compar_krigin},
is a sufficient condition for the NEB property:
\begin{assum}
  \label{assum:spectral-cond}
  The process $\xi$ is second-order stationary and has spectral
  density~$S$, with the property that $S^{-1}$ has at most polynomial
  growth.
\end{assum}
In other words, Assumption~\ref{assum:spectral-cond} means that there
exist $C>0$ and $r \in \NNN^{*}$ such that $S(u)(1+|u|^r) \;\geq\; C$,
almost everywhere on~$\RR^d$. Note that this is an assumption on~$k$, which
prevents it from being too regular.  In particular, the
so-called \emph{Gaussian covariance},
\begin{equation}
  \label{eq:Gaussian-RBF-cov}
  k(x,y) = s^2\, e^{- \alpha\, \ns{x-y}^2},
  \qquad s >0,\; \alpha > 0,
\end{equation}
does not satisfy Assumption~\ref{assum:spectral-cond}. In fact, and
this is the second part of our contribution, we shall show that $\xi$
with covariance function~(\ref{eq:Gaussian-RBF-cov})
does not possess the NEB property.
Assumption~\ref{assum:spectral-cond} still allows consideration of a large class
of covariance functions, which includes the class of (non-Gaussian)
exponential covariances
\begin{equation}
  \label{eq:exponential-cov}
  k(x,y) = s^2\, e^{- \alpha\, \ns{x-y}^\beta},
  \qquad s >0,\; \alpha > 0,\; 0<\beta<2\,,    
\end{equation}  
and the class of Matérn covariances~\citep[popularized by][]{Ste99}.

 To summarize, the main result of
this section is:
\begin{proposition}\label{prop:NEB}\mbox{}
  \vspace{-3pt} 
  \begin{enumerate}[i) ]
  \item\label{prop:suff-cond-NEB} If
    Assumption~\ref{assum:spectral-cond} holds, then $\xi$ has the NEB
    property.
  \item\label{prop:GRBF} If $\xi$ has the Gaussian covariance given
    by~(\ref{eq:Gaussian-RBF-cov}), 
    then~$\xi$ does not possess
    the NEB property.
  \end{enumerate}
\end{proposition}
The proof of Proposition~\ref{prop:NEB} is given in
Section~\ref{proof:prop:NEB}. To the best of our knowledge, finding
necessary and sufficient conditions for the NEB property---in other
words, solving Problem~$1$---is still an open problem.

\section{Pointwise consistency for continuous sample paths}
\label{sec:falsethm}

An important question related to Problem~$2$ is to know whether the
set~$\GG$ contains the set~$C(\RR^d)$ of all continuous functions.
\citet[Lemma 2.1]{ yakowitz85:_compar_krigin} claim, but fail to
establish, the following:
\begin{claaa}
  \label{claim:falseresult}
  Let Assumption~\ref{assum:spectral-cond} hold.  Assume that $\{x_n,\,
  n\geq 1\}$ is bounded, and denote by~$\XX_0$ its (compact) closure
  in~$\RR^d$.  Then, if $x \in \XX_0$,
  $$
  \forall f \in C(\RR^d)\,, \quad\lim_{n\to \infty} \scd{\PredMes{n}{x}, f} = f(x)\,.
  $$
\end{claaa}
Their incorrect proof has two parts, the first of which is correct; it
says in essence that, if $x \in \XX_0$ (i.e., if $x$ is adherent to~$\{
x_n,\, n\geq 1 \}$), then
\begin{equation}
  \label{eq:consist-rapid-decreas}
  \forall f \in \mathcal{S}(\RR^d), \quad\lim_{n\to \infty}
  \scd{\PredMes{n}{x}, f} = f(x)\,,  
\end{equation}
where $\mathcal{S}(\RR^d)$ is the vector space of rapidly decreasing
functions\footnote{Recall that $\mathcal{S}(\RR^d)$ corresponds to those
  $f\in C^{\infty}(\RR^d)$ for which
  \vspace{-4pt} $$
  \sup_{\abs{\nu} \leq N}\,\sup_{x\in\RR^d}~ (1+\abs{x}^{2})^N
  \abs{(D^{\nu} f)(x)} < \infty \vspace{-6pt}
  $$
  for $N=0,1,2,\ldots$, where $D^{\nu}$ denotes differentiation of order $\nu$.}.
In fact, this result stems from the weak convergence
result~(\ref{weak-Hstar-const}), once it has been remarked
that\footnote{Indeed, under Assumption~\ref{assum:spectral-cond}, we have $\forall f\in\mathcal{S}(\RR^d)$,
  \vspace{-4pt} $$
  \ns{f}_{\HH}^2
  = \frac{1}{(2\pi)^d}\int_{\RR^d} \left|\tilde f(u)\right|^2 S(u)^{-1}du 
  \leq \frac{1}{C\,(2\pi)^d}\int_{\RR^d} \left|\tilde f(u)\right|^2\, \left(1+|u|^r\right)\, du 
  < +\infty\,,\vspace{-6pt}
  $$
  where $\tilde f$ is the Fourier transform of $f$ \citep[see,
  e.g.,][]{wu93:_local}.} $\mathcal{S}(\RR^d) \subset \HH$ under  Assumption~\ref{assum:spectral-cond}.

The second part of the proof of~Claim~\ref{claim:falseresult} is flawed
because the extension of the convergence result from $\mathcal{S}(\RR^d)$ 
to $C(\RR^d)$, on the ground that $\mathcal{S}(\RR^d)$ is dense in~$C(\RR^d)$
for the topology of the uniform convergence on compact sets, does not work as
claimed by the authors. To get an insight into this, let $f\in C(\RR^d)$, and
let $(\phi_k) \in \mathcal{S}(\RR^d)^{\NNN}$ be a sequence that converges to
$f$ uniformly on~$\XX_0$.  Then we can write
\begin{align*}
  \left| \scd{\PredMes{n}{x}, f} - f(x) \right| 
  & \;\leq\; 
  \left| \scd{\PredMes{n}{x}, f - \phi_k} \right|
  + \left| \scd{\PredMes{n}{x} - \delta_x, \phi_k} \right|
  + \left| \phi_k(x) - f(x) \right| \\
  & \;\leq\; \left( 1 + 
    \left\lVert \PredMes{n}{x} \right\rVert_{\rm TV}
  \right)\, \sup_{\XX_0} \left| f - \phi_k \right|
  \,+\, \left| \scd{\PredMes{n}{x} - \delta_x, \phi_k} \right| \,,
\end{align*}
where $\ns{\PredMes{n}{x}}_{\rm TV} \defeq \sum_{i=1}^n \abs{\lambda^i(x;\x_n)}$
is the total variation norm of~$\PredMes{n}{x}$, also called the \emph{Lebesgue constant}
(at $x$) in the literature of approximation theory. If we assume that
the Lebesgue constant is bounded by~$K>0$, then we get,
using~(\ref{eq:consist-rapid-decreas}),
\begin{equation*}
  \limsup_{n\to \infty}\, \left| \scd{\PredMes{n}{x}, f} - f(x) \right| 
  \;\leq\; \left( 1 + K \right)\; \sup_{\XX_0}\,  \left| f - \phi_k \right|
  \;\xrightarrow[k\to\infty]{}\; 0 \,.
\end{equation*}
Conversely, if the Lebesgue constant is not bounded, the
Banach-Steinhaus theorem asserts that there exists a dense subset~$G$ of
$\left( C(\RR^d), \left\lVert \dotvar \right\rVert_\infty \right)$ such that, for all $f \in G$, 
$\sup_{n\geq1} \abs{\scd{\PredMes{n}{x}, f}} = +\infty$ \citep[see,
e.g.,][Section~5.8]{rundin87:_real}.

Unfortunately, little is known about Lebesgue constants in the literature
of kriging and kernel regression.  To the best of our knowledge, whether
the Lebesgue constant is bounded remains an open problem---although
there is empirical evidence
in~\cite{deMarchi08:_stabil_of_kernel_based_inter} that the Lebesgue
constant could be bounded in some cases.

Thus, the best result that we can state for now is a fixed version of
\cite{yakowitz85:_compar_krigin}, Lemma 2.1.
\begin{theorem}
  \label{thm:yako_corrected}
  Let Assumption~\ref{assum:spectral-cond} hold. Assume that $\{x_n,\, n\geq
  1\}$ is bounded, and denote by~$\XX_0$ its (compact) closure
  in~$\RR^d$.  Then, for all $x \in \XX_0$, the following assertions are
  equivalent:
  \begin{enumerate}[i) ]
  \item $\forall f \in C(\RR^d)$, $\lim_{n\to \infty} \scd{\PredMes{n}{x}, f} =
    f(x)$,
  \item the Lebesgue constant at~$x$ is bounded.
  \end{enumerate}
\end{theorem}

\section{Proof of Proposition~\ref{prop:NEB}}
\label{proof:prop:NEB}

Assume that~$x \in \RR^d$ is not adherent to~$ \{ x_n,\, n\!\geq 1
\}$. Then, there exists a $C^\infty(\RR^d)$ compactly supported function
$f$ such that $ f(x) \neq 0$ and $f(x_i) = 0$, $\forall i \in \{1,
\ldots, n \}$.  For such a function, the quantity $\scd{\PredMes{n}{x},
  f}$ cannot converge to~$f(x)$ since
\begin{equation*}
  \scd{\PredMes{n}{x}, f} \;=\;
  \sum_{i=1}^n \PredCoeff{i}{n}{x}\, f(x_i) \;=\; 0 \;\neq\; f(x) \,. 
\end{equation*}
Under Assumption~\ref{assum:spectral-cond}, $\mathcal{S}(\RR^d) \subset
\HH$, as explained in Section~\ref{sec:falsethm}. Thus, $f \in
\HH$; and it follows that $\PredMes{n}{x}$ cannot converge (weakly, hence
strongly) to~$\delta_x$ in~$\HH^*$. This proves the first assertion of
Proposition~\ref{prop:NEB}.

\medbreak
In order to prove the second assertion, pick any
sequence~$(x_n)_{n\geq1}$ such that the closure $\XX_0$ of~$\{ x_n,\, n\geq
1\}$ has a non-empty interior. We will show that $\sigma²(x;\x_n) \to 0$ for
\emph{all} $x \in \RR^d$. Then, choosing $x \not\in \XX_0$ proves the claim.

\def \LL {L^2 \left(\Omega, \mathcal{A}, \Pr \right)}

Recall that $\hat \xi(x;\x_n)$ is the orthogonal projection of~$\xi(x)$ onto
$\Span\{\xi(x_i),i=1,\ldots ,n\}$ in~$\LL$. Using the fact that the mapping $\xi(x)
\mapsto k(x,\cdot)$ extends linearly to an isometry\footnote{often referred to as
  Loève's isometry \citep[see, e.g.,][]{lukic01:_stoch_hilber}} from
$\overline\Span\{ \xi(y),\, y\in\RR^d \}$ to $\HH$, we get that
\begin{equation*}
  \sigma(x;\x_n) \;=\; \norm{ \xi(x)  - \hat\xi(x;\x_n) }
  \;=\; \mathop{d_\HH}\! \left( k(x,\cdot),\, H_n \right) \,,
\end{equation*}
where $d_{\HH}$ is the distance in $\HH$, and $H_n$ is the subspace
of~$\HH$ generated by~$k(x_i,\cdot)$, $i=1,\ldots,n$. Therefore
\begin{equation*}
  \lim_{n\to \infty} \sigma(x;\x_n) \;=\;
  \lim_{n\to \infty}  d_\HH\! \left( k(x,\cdot),\, H_n \right) \;=\;
  d_\HH\! \left( k(x,\cdot),\, H_\infty \right)\,,  
\end{equation*}
where $H_\infty = \overline{\cup_{n\geq1} H_n }$. Any function $f \in H_\infty^\perp$
satisfies $f(x_i) = \left( f,\, k(x_i,\cdot) \right) = 0$ and therefore
vanishes on~$\XX_0$, since~$\HH$ is a space of continuous functions.
Corollary~3.9 of \citet{steinwart:2006:explicit} leads to the conclusion that
$f=0$ since~$\XX_0$ has a non-empty interior. We have proved that $H_\infty^\perp
= \{0\}$, hence that $H_\infty = \HH$ since~$H_\infty$ is a closed subspace.
As a consequence, $\lim_{n\to \infty} \sigma(x;\x_n) = d_\HH\! \left( k_x,\, H_\infty
\right) = 0$, which completes the proof.
\qed

\bibliography{refs}

\end{document}